\documentclass[12pt]{amsart}

\usepackage{amssymb}
\usepackage{amsthm}

\usepackage[usenames]{color}

\usepackage[margin=1in]{geometry}
\usepackage[foot]{amsaddr}

\usepackage{hyperref}

\theoremstyle{plain}
\newtheorem{proposition}{Proposition}
\newtheorem{theorem}[proposition]{Theorem}
\newtheorem{lemma}[proposition]{Lemma}
\newtheorem{corollary}[proposition]{Corollary}
\theoremstyle{definition}

\newtheorem{definition}[proposition]{Definition}

\theoremstyle{remark}
\newtheorem{remark}[proposition]{Remark}

\DeclareMathOperator{\tr}{tr}
\DeclareMathOperator{\supp}{supp}

\DeclareMathOperator{\C}{\mathcal{C}}

\DeclareMathOperator{\WW}{\mathcal{W}}

\DeclareMathOperator{\RR}{\mathbb{R}}

\DeclareMathOperator{\JJ}{\mathbb{J}}

\newcommand{\abs}[1]{\left|#1\right|}

\newcommand{\norm}[1]{\left\|#1\right\|}
\newcommand{\wt}[1]{\widetilde{#1}}
\newcommand{\ip}[1]{\left\langle #1\right\rangle}

\begin{document}

\title{Compact asymptotically harmonic manifolds}
\author{Andrew M. Zimmer}\address{Department of Mathematics, University of Michigan, Ann Arbor, MI 48109.}
\email{aazimmer@umich.edu}
\date{\today}

\begin{abstract}
A complete Riemannian manifold without conjugate points is called asymptotically harmonic if the mean curvature of its horospheres is a universal constant. Examples of asymptotically harmonic manifolds include flat spaces and rank one locally symmetric spaces of noncompact type. In this paper we show that this list exhausts the compact asymptotically harmonic manifolds under a variety of assumptions including nonpositive curvature or Gromov hyperbolic fundamental group. We then present a new characterization of symmetric spaces amongst the set of all visibility manifolds.
\end{abstract}

\maketitle

\section{Introduction}

A complete Riemannian manifold $X$ is called \emph{harmonic} if about any point the mean curvature of a geodesic sphere of sufficiently small radius is constant. Examples of harmonic manifolds include flat spaces and rank one locally symmetric spaces. In fact Szab\'o~\cite{szabo90} showed that any harmonic manifold with a conjugate point is a rank one locally symmetric space of compact type. If $X$ is a simply connected harmonic manifold without conjugate points, Szab\'o~\cite{szabo90} observed that $X$ is also ``globally'' harmonic: the mean curvature of any geodesic sphere is constant and this constant only depends on the radius of the sphere. One can then consider the so-called \emph{asymptotically harmonic manifolds}, these are the complete Riemannian manifolds without conjugate points such that the mean curvature of their horospheres is a universal constant. By Szab\'o's observation, every harmonic manifold without conjugate points is asymptotically harmonic and it is natural to ask if the class of asymptotically harmonic manifolds can be characterized. 

In a simply connected complete Riemannian manifold $X$ without conjugate points the horosphere $H_v$ based at a vector $v$ in the unit tangent bundle $SX$ is defined to be the zero set of the horofunction function
\begin{align*}
b_v(x) = \lim_{t\rightarrow \infty} d(\gamma_v(t),x)-t.
\end{align*}
Hence, a complete Riemannian manifold $M$ with universal Riemannian cover $X$ is called \emph{asymptotically harmonic} if there exists $\alpha \in \RR$ such that for all $v \in SX$ the horofunction $b_v$ is $C^2$ and $\Delta b_v \equiv \alpha$. The main purpose of this paper is to characterize these manifolds under a variety of additional assumptions. 

\begin{theorem}
\label{thm:main1}
Suppose $M$ is a compact asymptotically harmonic manifold with universal Riemannian cover $X$. If any of the following holds 
\begin{enumerate}
\item there exists a vector $v \in S_pX$ such that the endomorphism 
\begin{align*}
\nabla^2 b_v(p) + \nabla^2 b_{-v}(p):T_{p} X \rightarrow T_{p} X
\end{align*}
has corank one,
\item $M$ has no focal points or less generally $M$ has nonpositive curvature,
\item $X$ is Gromov hyperbolic,
\item $X$ has purely exponential volume growth: let $h_{vol}$ be the volume growth entropy of $X$ then for each $p \in X$ there exists a constant $C>0$ such that for all $R \geq 1$:
\begin{align*}
\frac{1}{C} e^{h_{vol}R} \leq \text{vol}_X(B_R(p)) \leq C e^{h_{vol}R},
\end{align*}
\end{enumerate}
then $M$ is either flat or a rank one locally symmetric space of noncompact type.
\end{theorem}

\begin{remark}
If $M$ has nonpositive curvature, the condition in (1) is equivalent to $M$ having rank one in the sense of Ballmann, Brin, Eberlein~\cite{BBE85}.
\end{remark}

It is well known that if $M$ is a compact asymptotically harmonic manifold with negative sectional curvature, then $M$ is a rank one locally symmetric space. The proof of this fact is long and difficult. Foulon and Labourie~\cite{FL92} proved for any such manifold the stable and unstable foliations  of the geodesic flow are $C^\infty$. A deep rigidity result of Benoist, Foulon, and Labourie~\cite{BFL92} then implies that the geodesic flow on $SM$ is $C^\infty$ conjugate to the geodesic flow on a rank one symmetric space $N$. Finally Besson, Courtois, and Gallot's~\cite{BCG95} resolution of the minimum entropy conjecture implies that $M$ is isometric to the rank one symmetric space $N$. 

As observed by Knieper~\cite[Theorem 3.6]{knieper11}, these arguments actually show that any compact asymptotically harmonic manifold whose geodesic flow is Anosov is a rank one locally symmetric space. The main strategy in the proof of Theorem~\ref{thm:main1} is to show that the geodesic flow is Anosov and then apply the results mentioned above.

Theorem~\ref{thm:main1} should be compared to a recent result of Knieper concerning harmonic manifolds.

\begin{theorem}
\cite{knieper11}
Let $X$ be a non-compact complete harmonic manifold, if any of the following holds 
\begin{enumerate}
\item there exists a vector $v \in S_pX$ such that the endomorphism 
\begin{align*}
\nabla^2 b_v(p) + \nabla^2 b_{-v}(p):T_{p} X \rightarrow T_{p} X
\end{align*}
has corank one, 
\item $X$ has no focal points,
\item $X$ is Gromov hyperbolic,
\item $X$ has purely exponential volume growth,
\end{enumerate}
then either $X$ is flat or the geodesic flow on $SX$ is Anosov with respect to the Sasaki metric. In the latter case, if $X$ has a compact Riemannian quotient then $X$ is a rank one symmetric space of noncompact type.
\end{theorem}

\subsection{Applications:} Theorem~\ref{thm:main1} could be viewed as an unsurprising extension of Knieper's work. However the harmonic condition is very strong and difficult to establish. The asymptotically harmonic condition is also strong, but can be established in some cases using results of Ledrappier. Delaying definitions, in Section 5 results of Ledrappier will be interpreted (and weakened) as:

\begin{theorem}\cite{Led10}
\label{thm:led_basic}
Suppose $M$ is a compact Riemannian manifold without conjugate points. Let $X$ be the universal Riemannian cover of $M$ with deck transformations $\Gamma=\pi_1(M) \subset \text{Isom}(X)$. If $4\lambda_{min}=h_{vol}^2>0$ then there exists a $\Gamma$-Patterson-Sullivan measure $\{ \nu_x : x \in X\}$ on the Busemann boundary $\partial \hat{X}$ such that for $\nu_x$-almost every $\xi$, $\xi \in C^\infty(X)$ and $\Delta \xi \equiv h_{vol}$.
\end{theorem}

\begin{remark}
We delay definitions, but remark that for a manifold without conjugate points there is a fixed point $o \in X$ such that $\{ b_v-b_v(o) : v \in S X\} \subset \partial \hat{X}$. In particular, the conclusion of Theorem~\ref{thm:led_basic} implies that $M$ is asymptotically harmonic if each measure $\nu_x$ has full support.
\end{remark}

The parameter $h_{vol}$ is the volume growth entropy of $X$ and $\lambda_{min}$ is the bottom of the spectrum of the Laplace-Beltrami operator on $X$: 
\begin{align*}
\lambda_{min}(X) = \lambda_{min} = \inf\left\{\frac{\int_X \norm{\nabla f}^2 dx}{\int_X \abs{f}^2 dx} : f \in C^{\infty}_K(X)\right\}.
\end{align*}
Using smooth approximations of $e^{-sd(x,o)}$ as test functions for $2s > h_{vol}$, we see that 
\begin{align}
\label{eq:inq_vol_eigen}
4\lambda_{min} \leq h_{vol}^2.
\end{align}

In situations where one understands the Busemann boundary and the Patterson-Sullivan measures, Theorem~\ref{thm:main1} and Theorem~\ref{thm:led_basic} can be used to establish rigidity results. For example, Theorem~\ref{thm:main1} was used by Ledrappier and Shu in their proof of the following.

\begin{theorem}\cite[Theorem 1.1]{LS2012}\label{LSresult}
Suppose $M$ is a compact Riemannian manifold without focal points. Then the following are equivalent:
\begin{enumerate}
\item $M$ is a locally symmetric space,
\item $4\lambda_{min}=h_{vol}^2$,
\end{enumerate}
\end{theorem}

\begin{remark}
Ledrappier and Shu also proved rigidity results involving other parameters related to random walks on the universal of cover of $M$ (namely, the linear drift and stochastic entropy). We briefly outline their strategy to prove Theorem~\ref{LSresult}: they begin by reducing to the ``rank one'' case using a recent rigidity result of Watkins~\cite{watkins11}. Next they develop a theory of harmonic measures for rank one manifolds without focal points. Using this theory and previous results of Ledrappier~\cite{Led10} they deduce that $M$ is asymptotically harmonic. Finally, Theorem~\ref{thm:main1} is used to show that $M$ is locally symmetric.
\end{remark}

This result can be seen as a generalization of an old result of Ledrappier. Before compact asymptotically harmonic manifolds of negative curvature were classfied as rank one locally symmetric spaces, Ledrappier~\cite[Theorem 1]{led90} provided a number of equivalent formulations of the definition of asymptotically harmonic manifolds in negative curvature. Using the classification, Ledrappier's work provides a characterization of symmetric spaces in negative curvature. 

\begin{theorem}\cite[Theorem 1]{led90}
Let $M$ be a compact Riemannian manifold with negative sectional curvature and let $X$ be the universal Riemannian cover of $M$. Then the following are equivalent:
\begin{enumerate}
\item $X$ is a rank one symmetric of noncompact type,
\item $M$ is asymptotically harmonic,
\item each Busemann function $b_v$ has constant Laplacian and $\Delta b_v \equiv h_{vol}$,
\item $4\lambda_{min}=h_{vol}^2$,
\item the Patterson-Sullivan and harmonic measures on $X(\infty)$ coincide.
\end{enumerate}
\end{theorem} 

Motivated by this theorem, we obtain a new characterization of symmetric spaces among the so called ``visibility'' manifolds. This class of manifolds was originally defined by Eberlein and O\'Neil~\cite{EO73}. Ruggiero~\cite[Theorem 6.8]{ruggiero07} showed that visibility manifolds with a compact quotient are exactly the Gromov hyperbolic manifolds whose geodesic diverge.

In the negative curvature setting, the harmonic measures naturally arise from the identification of the Martin boundary and geometric boundary at infinity due to Anderson and Schoen~\cite{AS85}. For a general manifold, we follow Ledrappier~\cite{Led10} and consider harmonic measures on the laminated space $X_M = (X \times \partial \hat{X})/\Gamma$. Again delaying definitions we will prove the following:

\begin{theorem} 
\label{thm:visib}
Suppose $M$ is a compact Riemannian manifold without conjugate points. Let $X$ be the universal Riemannian cover of $M$ with deck transformations $\Gamma=\pi_1(M) \subset \text{Isom}(X)$. If $X$ is a visibility manifold, then the following are equivalent
\begin{enumerate}
\item $X$ is a rank one symmetric space of noncompact type, 
\item $X$ is asymptotically harmonic,
\item each Busemann function $b_v$ is $C^2$ and $\Delta b_v \equiv h_{vol}$,
\item there exists a function $f:X \rightarrow \RR$ that is 1-Lipschitz and has $\Delta f \geq h_{vol}$ (in the sense of distributions),
\item $4\lambda_{min} = h_{vol}^2$,
\item there exists a $\Gamma$-Patterson-Sullivan measure $\{ \nu_x : x \in X\}$ on $\partial \hat{X}$ such that the measure
\begin{align*}
d\wt{m} = dx \times d\nu_x(\xi)
\end{align*}
on $X \times \partial \hat{X}$ descends to a harmonic measure on the laminated space $X_M$.
\end{enumerate}
\end{theorem}

\begin{remark}
For a general compact manifold with non-compact universal cover, Ledrappier~\cite{Led10} has essentially shown that (5) and (6) are equivalent. The implication (4) implies (5) is due to Grigor\'yan~\cite[Theorem 11.17]{grig09}.
\end{remark}

\begin{remark}
Eberlein~\cite[Theorem 2]{eberlein72} has proven the following: suppose $(M,g_0)$ is a compact manifold without conjugate points and $(X,\wt{g}_0)$ is the universal Riemannian cover of $(M,g_0)$. Let $g_1$ be any other metric on $M$ without conjugate points and let $\wt{g}_1$ be the metric on $X$ induced from $g_1$. With this notation: $(X,\wt{g}_0)$ is a visibility manifold if and only if $(X,\wt{g}_1)$ is a visibility manifold.
\end{remark}

\subsection{Some History:} If $M$ is not assumed to be compact then there exist nonsymmetric homogeneous Hadamard manifolds, namely the Damek-Ricci spaces, which are asymptotically harmonic ~\cite{DR92}. Asymptotically harmonic manifolds without conjugate points have been classified in dimension 3~\cite{HKS07,SS08,shah11} and in the Einstein, homogeneous case by Heber~\cite{heber06}. All dimension three asymptotically harmonic manifolds are either flat or hyperbolic and any Einstein, homogeneous asymptotically harmonic manifold is either flat, a rank one symmetric space of noncompact type, or a nonsymmetric Damek-Ricci space. As all these examples have nonpositive curvature a theorem of Azencott and Wilson~\cite{AW76} implies that a nonsymmetric Damek-Ricci space does not admit compact quotients. In particular, it seems reasonable to conjecture that every asymptotically harmonic manifold with compact quotient is either flat or a rank one symmetric space.

\section{Preliminaries}

Every Riemannian manifold $M$ considered here will be complete, $SM$ will denote the unit tangent bundle, and $g^t$ the geodesic flow on $SM$. For a vector $v \in SM$, $\gamma_v : \RR \rightarrow M$ will denote the geodesic with $\gamma_v^\prime(0)=v$.

\subsection{Tensors along geodesics:} Given a Riemannian manifold $X$ and a geodesic $\gamma: I \rightarrow X$, let 
\begin{align*}
N_{\gamma} = \{ w \in T_{\gamma(t)} X: g(w,\gamma^\prime(t))=0\}
\end{align*}
be the normal bundle of $\gamma$. A (1,1)-tensor along $\gamma$ is a smooth bundle endomorphism of $N_{\gamma}$, i.e. a smooth map
\begin{align*}
t \in I \rightarrow Y(t) \in \text{End}(\gamma^\prime(t)^{\bot}).
\end{align*}
Given a smooth (1,1)-tensor $Y$ we can take derivatives using the Levi-Civita connection to obtain a new (1,1)-tensor: $Y^\prime = \nabla_{\gamma^\prime(t)} Y$. This differentiation is easily realized using parallel vector fields: if $x_t,y_t$ are parallel vector fields along $\gamma$, then $g(x_t,Y^\prime(t)y_t) = \frac{d}{dt}g(x_t,Y(t)y_t)$.

Let $R$ be the curvature tensor on $X$. An example of a (1,1)-tensor that we will use frequently is the Riemannian curvature tensor $t \rightarrow R_{\gamma}(t)$ given by
\begin{align*}
R_{\gamma}(t)x = R(x,\gamma^\prime(t))\gamma^\prime(t).
\end{align*}
We will usually drop the $\gamma$ and just write $R(t)$.
  
An useful class of (1,1)-tensors are the so called Jacobi tensors. A (1,1)-tensor $\JJ$ along a geodesic $\gamma:\RR \rightarrow X$ is called a \emph{Jacobi tensor} if 
\begin{align*}
\JJ^\prime(t) + R(t)\JJ(t) = 0.
\end{align*}
If $x_t$ is a parallel vector field along $\gamma$ orthogonal to $\gamma^\prime(t)$ then $\JJ(t)x_t$ will be a Jacobi field along $\gamma$.

\subsection{Asymptotically harmonic manifolds:} In this subsection we will recall a useful result of Ranjan and Shah. Suppose $M$ is an asymptotically harmonic manifold with universal Riemannian cover $X$. Then for $v \in SX$ the Busemann function
\begin{align*}
b_v(x) = \lim_{t\rightarrow \infty} d(\gamma_v(t),x)-t.
\end{align*}
is $C^2$ and $\Delta b_v \equiv \alpha$ for some universal constant $\alpha$.

\begin{remark}
In Heber's paper~\cite{heber06} asymptotically harmonic manifolds are defined in terms of the unstable/stable Riccati solutions. If $X$ has no focal points or nonpositive curvature, this will be equivalent to the definition we use. Heber proved that his definition of asymptotically harmonic implies that all Busemann functions $b_v$ have constant Laplacian~\cite[Remark 2.2(c)]{heber06}. Using the weaker definition makes the proof of Theorem~\ref{thm:main1} slightly more technical, but makes it easier to establish that a manifold is asymptotically harmonic in Theorem~\ref{thm:visib}. 
\end{remark} 

\begin{theorem}\cite[Theorem 5.1]{rs2003}\label{thm:cont}
Let $X$ be a simply connected asymptotically harmonic manifold. Then
\begin{enumerate}
\item for $v \in SX$, $b_v \in C^\infty(X)$,
\item the map $v \in SX \rightarrow b_v \in C^\infty(X)$ is continuous.
\end{enumerate}
\end{theorem}

\begin{remark}
Here $C^{\infty}(X)$ is equipped with the topology of uniform covergence on compact sets: $f_n \rightarrow f$ in $C^\infty(X)$ if and only if for each $m \geq 0$, $\nabla^m f_n \rightarrow \nabla^m f$ converges locally uniformly.
\end{remark}

If $X$ is a general simply connected manifold with nonpositive curvature then $b_v$ will be $C^2$ and the map $v \in SX \rightarrow b_v \in C^2(X)$ will be continuous. In the general case of no conjugate points, the map $v \rightarrow b_v$ may not even be continuous in the $C^0$ topology. 

If $b_v$ is $C^1$, then it is a distance function and so the integral curves of $-\nabla b_v$ are geodesics. If $b_v$ is $C^2$, then the flow generated by the vector field $x \rightarrow -\nabla b_v(x)$ then yields a Jacobi tensor along the geodesic $\gamma_v$. In nonpositive curvature this Jacobi tensor will be the stable Jacobi tensor, but in the general case of no conjugate points the relationship between the two is unclear (in general $b_v$ will not be $C^2$). In the next three subsections we will prove some results relating these two tensors in the case of asymptotically harmonic manifolds.

\subsection{A Useful Endomorphism:}
Suppose $M$ is an asymptotically harmonic manifold with universal Riemannian cover $X$. Let $v \in S_pM$ and suppose $\wt{v} \in S_{\wt{p}}X$ is a lift of $v$. Then by Theorem~\ref{thm:cont} the function $b_{\wt{v}}$ is $C^2$. 
Let $B(\wt{v}) \in \text{End}(\wt{v}^{\bot})$ be the endomorphism defined by $Y \rightarrow \nabla_Y \nabla b_{\wt{v}}(\wt{p})=\nabla^2 b_{\wt{v}}(\wt{p})Y$. 
Notice that $\nabla_{\wt{v}} \nabla b_{\wt{v}}(\wt{p})=0$ and $\nabla^2 b_{\wt{v}}(\wt{p})$ is symmetric with respect to the Riemannian metric so $B(\wt{v})$ does indeed map $\wt{v}^{\bot}$ to  $\wt{v}^{\bot}$. 
Finally using the identification of $v^{\bot}$ and $\wt{v}^{\bot}$ given by our covering map, $B(\wt{v})$ determines an endomorphism $B(v):v^{\bot} \rightarrow v^{\bot}$ that does not depend on our choice of lift.

Notice that $B(v)$ is symmetric with respect to the Riemannian metric and $\tr B(v) = \tr \nabla^2 b_{\wt{v}}(\wt{p})=\Delta b_{\wt{v}}( \wt{p})=\alpha$. Further by Theorem~\ref{thm:cont}, the map 
\begin{align*}
v \in SM \rightarrow B(v) \in \cup_{w \in SM} \text{End}(w^{\bot})
\end{align*}
is continuous.

\begin{lemma}
Let $M$ be an asymptotically harmonic manifold. If $v \in SM$, then the path $t \rightarrow -B(g^t v)$ is a smooth (1,1)-tensor along the geodesic $\gamma_v$ and satisfies the Riccati equation
\begin{align}
\label{eq:riccati}
U^\prime +U^2 +R(t)=0
\end{align}
where $R(t) = R(\cdot,\gamma_v^\prime(t))\gamma_v^\prime(t)$ is the curvature tensor along $\gamma_v$. In particular, the solution $\JJ(t)$ to the differential equation
\begin{align*}
\JJ^\prime(t) & = -B(g^t v) \JJ(t) \\
\JJ(0) & =Id
\end{align*}
is a Jacobi tensor along the geodesic $\gamma_v$. 
\end{lemma}

\begin{proof}
By passing to the universal cover of $M$ we can suppose $M$ is simply connected. Fix a parallel vector field $y_t$ along the geodesic $\gamma_v(t)$ orthogonal to $\gamma_v^\prime(t)$. Let $U$ be the (1,1)-tensor along $\gamma_v$ given by $t \rightarrow -B(g^t v)$. By the symmetry of the connection and the fact that $\gamma_v^\prime(t)=-\nabla b_v(\gamma_v(t))$ we have
\begin{align*}
[y_t,-\nabla b_v] = \nabla_{y_t} (-\nabla b_v) - \nabla_{-\nabla b_v } (y_t) = -\nabla_{y_t} \nabla b_v = U(t)(y_t)
\end{align*} 
Again using the fact that $\gamma_v^\prime(t)=-\nabla b_v(\gamma_v(t))$ we have
\begin{align*}
R(t)y_t
&=R(y_t,\gamma_v^\prime(t))\gamma_v^\prime(t) \\
&= \nabla_{y_t} \nabla_{\gamma_v^\prime(t)} (\gamma_v^\prime(t))-\nabla_{\gamma_v^\prime(t)} \nabla_{y_t} (-\nabla b_v(\gamma_v(t))) - \nabla_{[y_t,-\nabla b_v]} (-\nabla b_v(\gamma_v(t))) \\
&= \nabla_{\gamma_v^\prime(t)} \nabla_{y_t} \nabla b_v(\gamma(t))+ \nabla_{U(y_t)} \nabla b_v((\gamma_v(t))) \\
&= -U^\prime(t)(y_t) - U^2(y_t)
\end{align*}
And hence the path $t \rightarrow U(t)$ satisfies the Riccati equation.
\end{proof}

\begin{lemma}
\label{lem:Buse_ineq}
Let $M$ be a asymptotically harmonic manifold with universal constant $\alpha \equiv \Delta b_v$. If $v \in SM$, then $B(-v)+B(v) \geq 0$. In particular $\alpha \geq 0$.
\end{lemma}

\begin{proof}
By passing to the universal cover of $M$, we may suppose $M$ is simply connected. Suppose $v \in S_p M$ then $B(v) = \nabla^2 b_{v}(p)|_{v^{\bot}}$. Notice that,
\begin{align*}
b_{-v}^T(x)+b_{v}^T(x) = d(x,\gamma_v(T))+d(x,\gamma_{-v}(T))-2T \geq 0
\end{align*}
by the triangle inequality. So
\begin{align*}
b_{-v}(x) + b_{v}(x) = \lim_{T\rightarrow \infty} b_{-v}^T(x)+b_{v}^T(x) \geq 0.
\end{align*}
Further $b_{-v}(p) + b_{v}(p) =0$ and $\nabla b_{-v}(p) + \nabla b_{v}(p)=0$. So by the second derivative test applied at $x=p$ we have $\nabla^2 b_{-v}(p) + \nabla^2 b_v(p) \geq 0$.
\end{proof}

\subsection{The Stable and Unstable Riccati solutions} 

We begin by introducing the stable and unstable Riccati solutions. Suppose $M$ is a complete Riemannian manifold without conjugate points, let $v \in SM$ and consider the Jacobi tensor $\JJ_{v,T}$ along $\gamma_v$ such that $\JJ_{v,T}(0)=Id$ and $\JJ_{v,T}(T)=0$. Let $U^s_T(v)=\JJ_{v,T}^\prime(0)$. We then have the following.

\begin{proposition}\cite{green58}
\label{prop:riccati_basic}
With the notation above,
\begin{enumerate}
\item $\JJ_{v,T}$ converges to a Jacobi tensor $\JJ_v^s$ along $\gamma_v$,
\item $U^s_T(v)$ converges monotonically to an endomorphism $U^s(v)$ in the sense that $U^s_{T_2}(v)-U^s_{T_1}(v)$ is positive definite for all $T_2>T_1>0$,
\item $U^s(v) = (\JJ_v^s)^\prime(0)$ and $U^s(g^tv) = (\JJ_v^s)^\prime(t)\JJ_v^s(t)^{-1}$,
\item the (1,1)-tensor $t \rightarrow U^s(g^tv)$ satisfies the Riccati equation:
\begin{align*}
(U^s)^\prime+(U^s)^2+R=0
\end{align*}
\end{enumerate}
where $R(t) = R(\cdot, \gamma_v^\prime(t))\gamma_v^\prime(t)$ is the curvature tensor along $\gamma_v$.  
\end{proposition}

The tensor $\JJ^s_v$ is called the \emph{stable Jacobi tensor along $\gamma_v$} and the map $v \rightarrow U^s(v)$ is called the \emph{stable Riccati solution}. We also have the \emph{unstable Jacobi tensor along $\gamma_v$} given by $\JJ^u_v(t) = \JJ^s_{-v}(-t)$ and the \emph{unstable Riccati solution} given by $U^u(v) = -U^s(-v)$. Notice that the (1,1)-tensor $t \rightarrow U^u(g^tv)$ also satisfies the Riccati equation.

The stable and unstable Jacobi fields have the following geometric characterization due to Eberlein.

\begin{lemma}
\cite[Proposition 2.12, Corollary 2.14]{eberlein73I}\label{lem:kerV}
Let $M$ be a complete Riemannian manifold without conjugate points and sectional curvature bounded from below. Suppose $\gamma$ is a unit speed geodesic and $J(t)$ is a Jacobi field along $\gamma$ with $g(J(0),\gamma^\prime(0))=0$ then
\begin{enumerate}
\item if $\norm{J(t)}$ is bounded above for all $t \geq 0$ then $J^\prime(0)=U^s(v)J(0)$,
\item if $\norm{J(t)}$ is bounded above for all $t \leq 0$ then $J^\prime(0)=U^u(v)J(0)$.
\end{enumerate}
If, in addition, $M$ has no focal points then following are equivalent:
\begin{enumerate}
\item $\norm{J(t)}$ is bounded above for all $t \in \RR$,
\item $\norm{J(t)}$ is constant,
\item $J^\prime(0)=U^u(v)J(0)=U^s(v)J(0)$.
\end{enumerate}
\end{lemma}

\begin{remark}
The above assertions concerning manifolds without focal points follow from Eberlein's results, but complete details can be found in~\cite[Section 5]{eschen77}.
\end{remark}

\subsection{The connection between the stable, unstable Riccati solutions and the Busemann functions.} 

\begin{lemma}
\label{lem:ineq1}
Let $M$ be an asymptotically harmonic manifold. If $v \in SM$ then
\begin{enumerate}
\item $U^s(v) \leq -B(v)$,
\item $U^u(v) \geq B(-v)$.
\end{enumerate}
\end{lemma}

\begin{proof}
By passing to the universal cover of $M$, we may suppose $M$ is simply connected. If $v \in S_pM$ then $B(v) = \nabla^2 b_{v}(p)|_{v^{\bot}}$. By the definition of $U^u(v)$ it is enough to prove the first inequality. Observe that the integral curves of 
\begin{align*}
x \rightarrow b_v^T(x)=d(x,\gamma_v(T))-T
\end{align*}
on $X\setminus \{\gamma_v(T)\}$ are geodesics and these geodesics give rise to the Jacobi tensor $\JJ_{v,T}$ along $\gamma_v$. It is not hard to show that $ U^s_T(v)=-\nabla^2 b_v^T(p)|_{v^{\bot}}$. Further 
\begin{align*}
b_v(x) 
&= \lim_{t \rightarrow \infty} d(x,\gamma_v(t))-t \\
&\leq \lim_{t \rightarrow \infty} d(x,\gamma_v(T))+d(\gamma_v(T),\gamma_v(t))-t\\
&= d(x,\gamma_v(T))-T = b_v^T(x),
\end{align*}
$b_v^T(p)=b_v(p)=0$, and $\nabla b_v^T(p) = \nabla b_v(p)=-v$. Hence by the second derivative test applied to $b_v^T-b_v$ at $x =p$, we see that $\nabla^2 b_v^T(p) \geq \nabla^2 b_v(p)$. Then
\begin{equation*}
U^s(v) = \lim_{T \rightarrow +\infty} U^s_T(v) \leq -\nabla^2 b_v(p)|_{v^{\bot}} = -B(v). \qedhere
\end{equation*}
\end{proof}

\begin{lemma}
\label{lem:rank_defn}
Let $M$ be an asymptotically harmonic manifold. Suppose $v \in SM$ then
\begin{enumerate}
\item for $T>0$, $U^u(v)-U^s_T(v) > U^u(v)-U^s(v) \geq 0$,
\item if $\det (B(-v)+B(v)) \neq 0$ then $\det (U^u(v)-U^s(v)) \neq 0$.
\end{enumerate}
\end{lemma}

\begin{remark} If $\det(U^u(v)-U^s(v))\neq 0$, Lemma~\ref{lem:kerV} implies that the vector space of bounded Jacobi fields along $\gamma_v$ is one dimensional and spanned by the Jacobi field $t \rightarrow \gamma^\prime(t)$.
\end{remark}

\begin{proof}
By Proposition~\ref{prop:riccati_basic} part (2), we have $U^u(v)-U^s_T(v) > U^u(v)-U^s(v)$. Also Lemma~\ref{lem:Buse_ineq} and Lemma~\ref{lem:ineq1} imply that
\begin{align*}
U^u(v)-U^s(v) \geq B(-v) + B(v) \geq 0
\end{align*}
and the lemma follows.
\end{proof}

In Section 4 we will prove the following useful fact.

\begin{proposition}
\label{prop:asym_constant}
Let $M$ be a compact asmptotically harmonic manifold with universal Riemannian cover $X$ and $\Delta b_v \equiv \alpha$ for all $v \in SX$. Then $\alpha = h_{vol}$ and $B(-v) = U^u(v)$ for almost every $v \in SM$ with respect to the Liouville measure.    
\end{proposition}

\section{Proof of Theorem~\ref{thm:main1}}

\subsection{Outline of Proof} As mentioned in the introduction the combined works of Foulon and Labourie~\cite{FL92}, Benoist, Foulon, and Labourie~\cite{BFL92}, and Besson, Courtois, and Gallot~\cite{BCG95} imply the following (see Knieper~\cite[Theorem 3.6]{knieper11} for some details).

\begin{theorem}\label{thm:AHanosov}
Suppose $M$ is a compact asymptotically harmonic manifold. If the geodesic flow on $SM$ is Anosov, then $M$ is a rank one locally symmetric space of noncompact type.
\end{theorem}
 
We will reduce each condition in Theorem~\ref{thm:main1} to the above theorem. Let $M$ be a compact Riemannian manifold. In subsection~\ref{subsec:rankone}, we will show that if $M$ is asymptotically harmonic and there exists a single $v \in SM$ with $\det(B(-v) + B(v)) \neq 0$ then a result of Eberlein implies that the geodesic flow is Anosov. Then by Theorem~\ref{thm:AHanosov} $(M,g)$ must be a locally symmetric space. In subsection~\ref{subsec:no_focal}, we will use a recent generalization due to Watkins~\cite{watkins11} of the Rank Rigidity Theorem. Using this we will show every higher rank asymptotically harmonic manifold without focal points is flat. 

In subsection~\ref{subsec:purelyExp} we will show that purely exponential volume growth implies the existence of a $v \in SM$ with $\det(B(-v) + B(v)) \neq 0$. Finally a theorem of Coornaert~\cite[Theorem 7.2]{coor93} implies that any Gromov hyperbolic space with compact quotient has purely exponential volume growth.

\subsection{The ``rank one'' case:}
\label{subsec:rankone}

We begin by stating a criterium for the geodesic flow being Anosov. Recall that for $v \in T_pM$ we have a canonical splitting 
\begin{align*}
T_v TM = T_p M \oplus T_pM
\end{align*}
where the first factor is the \emph{horizontal distribution} and the second is the \emph{vertical distribution}. The manifold $TM$ has a natural metric coming from this identification:
\begin{align*}
\left\langle (X_1,Y_1),(X_2,Y_2)\right\rangle = g(X_1,X_2)+g(Y_1,Y_2).
\end{align*}
If $v \in SM$, then 
\begin{align*}
T_v SM = \{ (X,Y) : Y \in v^{\bot}\}.
\end{align*}
Next define the \emph{stable and unstable Green subbundles} as
\begin{align*}
E^s(v) &= \{ (X,U^s(v)X) : X \in v^{\bot} \}, \\
E^u(v) &= \{ (X,U^u(v)X) : X \in v^{\bot} \}.
\end{align*}
If the geodesic flow is Anosov, the stable and unstable Green subbundles are in fact the stable and unstable distributions of the geodesic flow and 
\begin{align}
\label{eq:span}
T_v SM = E^s(v) \oplus E^u(v) \oplus Z
\end{align}
where $Z$ is the flow direction. Remarkably, Eberlien proved that if the Green subbundles span the tangent space as in ~\ref{eq:span} then the geodesic flow is Anosov. 

\begin{theorem}\cite{eberlein73I}
\label{thm:anosov}
Let $M$ be a compact Riemannian manifold without conjugate points. Then
$\det (U^u(v)-U^s(v)) \neq 0$ for all $v \in SM$
if and only if the geodesic flow on $SM$ is Anosov.
\end{theorem}

Assuming the curvature is bounded below and the spectrum of $U^u(v)-U^s(v)$ is uniformly bounded below by a positive number, Bolton~\cite{bolton79} later proved Theorem~\ref{thm:anosov} without the compactness assumption.

In this subsection we will prove the following.

\begin{proposition}\label{prop:rank_one}
Suppose $M$ is a compact asymptotically harmonic manifold. If there exists a $v \in SM$ with 
$
\det (B(-v)+B(v) )\neq 0
$
then for all $v \in SM$, $\det(U^u(v)-U^s(v)) \neq 0$. Hence the geodesic flow is Anosov and $M$ is a rank one locally symmetric space of noncompact type.
\end{proposition}

\begin{remark}
Recall that $B(v) \in \text{End}(v^{\bot})$ is the endomorphism obtained by lifting $v \in S_pM$ to $\wt{v} \in S_{\wt{p}}X$ and considering the map $B(\wt{v}):\wt{v}^{\bot} \rightarrow  \wt{v}^{\bot}$ given by $Y \rightarrow \nabla_Y \nabla b_{\wt{v}}(\wt{p})$. As $\nabla_{\wt{v}} \nabla b_{\wt{v}}(\wt{p}) =0$, $\det( B(-v)+B(v)) \neq 0$ if and only if the endomorphism $\nabla^2 b_{-\wt{v}}(\wt{p})+\nabla^2 b_{\wt{v}}(\wt{p}) : T_{\wt{p}} X \rightarrow T_{\wt{p}} X$ has corank one.
\end{remark}

As in~\cite{HKS07,knieper11, SS08} we consider the map 
\begin{align*}
v \in SM \rightarrow V(v)=B(-v)+B(v)\in \text{End}(v^{\bot}).
\end{align*}
Let $g^t:SM \rightarrow SM$ be the geodesic flow on $SM$. Because the (1,1)-tensors $t \rightarrow \mp B(\pm g^t v)$ solve the Riccati equation~\ref{eq:riccati}, for any $v \in SM$ the (1,1)-tensor $t \rightarrow V(g^t v)$ satisfies the differential equation
\begin{align*}
V^\prime = XV + VX
\end{align*}
where $X(t) = -\frac{1}{2}(B(-v)-B(v))$. 

We will next show that $v \rightarrow \det V(v)$ is invariant under the geodesic flow. The following lemma is given in~\cite{HKS07,SS08}, but for completeness we include the short proof.

\begin{lemma}
\cite{HKS07}
\label{lem:inv}
Let $M$ be an asymptotically harmonic manifold. Then the map $v \rightarrow \det(V(v))$ is invariant under the geodesic flow.
\end{lemma}

\begin{proof}
For $v \in SM$ we will show that the function $t \in \RR \rightarrow \det(V(g^t v))$ is constant. It is enough to prove this for those $v \in SM$ with $\det(V(v)) \neq 0$. In this case let $V(t) = V(g^t v)$, we then obtain
\begin{align*}
\frac{d}{dt} \log \det V = \tr \dot{V} V^{-1}=\tr(XV+VX)V^{-1} =2\tr X  =0
\end{align*}
as  $\tr B(w) = \Delta b_{\wt{w}} \equiv \alpha$ for any $w \in SM$.
\end{proof}

\begin{proposition}
\label{prop:hyperbolic_sets}
Let $M$ be a compact asymptotically harmonic manifold. If the set
\begin{align*}
\mathcal{C}_\epsilon= \{ v \in SM : \det V(v) \geq \epsilon\}
\end{align*} 
is nonempty, then $\mathcal{C}_{\epsilon}$ is a hyperbolic set for the geodesic flow: there exists $C,\lambda>0$ such that for all $v \in \mathcal{C}_{\epsilon}$, 
\begin{align*}
T_v SM = E^s(v) \oplus E^u(v) \oplus Z
\end{align*}
where $Z$ is the flow direction and 
\begin{align*}
\norm{D(g^t)_v W } & \leq C e^{-\lambda t}\norm{W} \text{ for $t \geq 0$ and $W \in E^s(v)$} \\ 
\norm{D(g^t)_v W } & \leq C e^{-\lambda t}\norm{W} \text{ for $t \leq 0$ and $W \in E^u(v)$}
\end{align*}
\end{proposition}

Assuming Proposition~\ref{prop:hyperbolic_sets} we can prove Proposition~\ref{prop:rank_one}:

\begin{proof}[Proof of Proposition~\ref{prop:rank_one}] We claim that for each $\epsilon >0$ the function $v \rightarrow\det V(v)$ is constant on each component of $\mathcal{C}_\epsilon$. By Proposition~\ref{prop:hyperbolic_sets}, $\mathcal{C}_{\epsilon}$ is a hyperbolic set for the geodesic flow $g^t$. Using Proposition 6.4.13 in~\cite{KH95} for each $v,w \in \mathcal{C}_\epsilon$ sufficiently close there exists $u \in SM$ such that 
\begin{align*}
d_{SM}(g^tv,g^tu) \rightarrow 0 \text{ and } d_{SM}(g^{-t}w,g^{-t}u) \rightarrow 0 \text{ as $t \rightarrow \infty$ }.
\end{align*}
Now as $SM$ is compact and $v \rightarrow \det V(v)$ is continuous and invariant under the geodesic flow $\det V(v) = \det V(u) = \det V(w)$. 

By hypothesis, the set $\mathcal{C}=\{ v \in SM : \det V(v) >0\}$ is nonempty. Using the above argument, we see that $v \rightarrow \det V(v)$ is locally constant on $\mathcal{C}$. As $SM$ is a second-countable space, this shows that $v \rightarrow \det V(v)$ can take on at most countable many values which means it must be constant and nonzero. So $V(v)=\det(B(-v)+B(v)) \neq 0$ for all $v \in SM$ so by Lemma~\ref{lem:rank_defn} $\det(U^u(v)-U^s(v)) \neq 0$  for all $v \in SM$, so by Theorem~\ref{thm:anosov} the geodesic flow is Anosov.  
\end{proof}

The proof of Proposition~\ref{prop:hyperbolic_sets} is essentially a condensed version of a proof due to Ruggiero~\cite[Chapter 3]{ruggiero07} of Eberlein's Theorem. 

Fix a compact asymptotically harmonic manifold $M$, a constant $\epsilon >0$, and the set $\C_{\epsilon} =\{ v \in SM : \det V(v)\geq \epsilon\}$. 

\begin{lemma}
With the notation above, $\C_\epsilon$ is compact and $g^t$ invariant. Further for any $v  \in \C_\epsilon$ the vector space of bounded Jacobi fields along $\gamma_v$ has dimension one and is spanned by the Jacobi field $t \rightarrow \gamma_v^\prime(t)$.
\end{lemma}

\begin{proof}
Theorem~\ref{thm:cont} implies that $v \rightarrow \det V(v)$ is continuous and hence $\C_\epsilon$ is closed (and compact). The $g^t$ invariance of $\C_\epsilon$ follows immediately from Lemma~\ref{lem:inv}. Finally the last assertion follows from Lemma~\ref{lem:kerV} and Lemma~\ref{lem:rank_defn}.
\end{proof}

If $v \in SM$ and 
\begin{align*}
V=(V_1,V_2) \in T_v SM = \{ (X,Y) : Y \in v^{\bot}\},
\end{align*}
then $D(g^t)_v V = (J(t),J^\prime(t))$ where $J$ is the Jacobi field along $\gamma_v$ with $J(0)=V_1$ and $J^\prime(0)=V_2$. The following lemma is a standard result for compact manifolds, see for instance~\cite[Proposition 2.7]{eberlein73I}.

\begin{lemma}
\label{lem:bound}
With the notation above, there exists $\kappa >0$ such that if $v \in SM$ and $J(t)$ is a Jacobi field along the geodesic $\gamma_v$ with $J(0)=0$ and $g(J^\prime(0),v)=0$ then $\norm{J^\prime(t)} \leq \kappa \coth(\kappa t) \norm{J(t)}$ for $t \geq 0$. In particular, if $J$ is a Jacobi field along the geodesic $\gamma_{v}$ and $J^\prime(0)= U^s(v)J(0)$ then $\norm{J^\prime(t)} \leq \kappa\norm{J(t)}$ for $t \geq 0$
\end{lemma}

\begin{proof}
The first assertion is~\cite[Proposition 2.7]{eberlein73I}. To see the second assertion, observe that $J$ is the pointwise limit of Jacobi fields $J_T$ with $J_T(0)=J(0)$ and $J_T(T)=0$.
\end{proof}

\begin{lemma}
\label{lem:bound2}
With the notation above, there exists a constant $L>0$ such that whenever $v \in \C_\epsilon$, $T>0$, and $J$ is a Jacobi field along the geodesic $\gamma_{v}$ with $J(0)=0$ and $g(v,J^\prime(0))=0$ then $\norm{J(t)} \leq L\norm{J(T)}$ for $t \in [0,T]$. In particular, if $J$ is a Jacobi field along the geodesic $\gamma_{v}$ and $J^\prime(0)= U^s(v)J(0)$ then $\norm{J(t)} \leq L\norm{J(0)}$ for $t \geq 0$
\end{lemma}

\begin{proof}
Suppose the first assertion is false, then there exists $v_n \in \C_\epsilon$, $T_n >0$, and Jacobi fields $J_n$ along $\gamma_{v_n}$ with $J_n(0)=0$, $g(v_n,J_n^\prime(0))=0$, and $\norm{J_n(T_n)}=1$ such that
\begin{align*}
\norm{J(s_n)} \geq n
\end{align*}
for some $s_n \in [0,T_n]$. We may assume that $s_n \in [0,T_n]$ is picked such that $\norm{J(s_n)}$ is maximal. Now by Lemma~\ref{lem:bound}, $s_n$ and $T_n-s_n$ must go to infinity as $n$ goes to infinity. Then 
\begin{align*}
Z_n(t)=\frac{1}{\norm{J_n(s_n)}}J_n(t+s_n)
\end{align*}
is a Jacobi field bounded above by one on the interval $[-s_n,T_n-s_n]$. By passing to a subsequence we may suppose that $g^{s_n}v_n \rightarrow v \in \C_\epsilon$, $Z_n(0) \rightarrow X \in v^{\bot}$, and $Z_n^\prime(0)\rightarrow Y \in v^{\bot}$.  Then if $J$ is the Jacobi field along $\gamma_v$ with $J(0)=X$ and $J^\prime(0)=Y$, the sequence $Z_n$ converges pointwise to $J$ and so $\norm{J(t)}\leq 1$ on $\RR$. However this is a contradiction because $g(X,v)=0$ and $v \in \C_\epsilon$.

To see the second assertion, observe that $J$ is the limit of Jacobi fields $J_T$ with $J_T(0)=J(0)$ and $J_T(T)=0$.
\end{proof}  

\begin{proof}[Proof of Proposition~\ref{prop:hyperbolic_sets}] It is enough to demonstrate the existence of $C, \lambda >0$ such that for all $v \in \C_\epsilon$ and $V \in E^s(v)$ we have
\begin{align*}
\norm{D(g^t)_v V} \leq Ce^{-\lambda t} \norm{V}
\end{align*}
for $t \geq 0$. Let 
\begin{align*}
f(t) = \sup\{\norm{ D(g^t)_vV} : v \in \C_\epsilon, V \in E^s(v), \norm{V}=1\}.
\end{align*}

We first claim that $\lim_{t \rightarrow  \infty} f(t) =0$. If $J$ is the Jacobi field long $\gamma_v$ with $(J(0),J^\prime(0))=V$ then
\begin{align*}
\norm{D(g^t)_v V}^2 = \norm{J(t)}^2 + \norm{J^\prime(t)}^2 \leq (1+\kappa^2)\norm{J(t)}^2,
\end{align*} 
by Lemma~\ref{lem:bound}, so it is enough to show that $\norm{J(t)}$ converges uniformly to $0$ for all Jacobi fields along geodesics $\gamma_v$ with $v \in \C_\epsilon$, $(J(0),J^\prime(0)) \in E^s(v)$, and $\norm{J(0)}^2+\norm{J^\prime(0)}^2=1$. 

Suppose not then there exists $\epsilon \in (0,1)$ and sequences $v_n \in \C_\epsilon$, $t_n \geq 0$, and $(X_n,Y_n) \in E^s(v_n)$ such that $t_n \rightarrow \infty$, $\norm{(X_n,Y_n)}=1$, and the Jacobi field $J_n$ with $J_n(0)=X_n$ and $J_n^\prime(0)=Y_n$ has $\norm{J_n(t_n)} \geq \epsilon$.   

Then the Jacobi fields 
\begin{align*}
Z_n(t) = \frac{1}{\norm{J_n(t_n)}}J_n(t+t_n)
\end{align*}
are bounded by $\epsilon^{-1}L$ on the intervals $[-t_n,\infty)$. By passing to a subsequence we may suppose that $g^{t_n}v_n \rightarrow v \in \C_\epsilon$, $Z_n(0) \rightarrow X \in v^{\bot}$, and $Z_n^\prime(0)\rightarrow Y \in v^{\bot}$.  Then if $J$ is the Jacobi field along $\gamma_v$ with $J(0)=X$ and $J^\prime(0)=Y$, the sequence $Z_n$ converges pointwise to $J$ and so $\norm{J(t)} \leq \epsilon^{-1}L$ on $\RR$. However this is a contradiction because $g(X,v)=0$ and $v \in \C_\epsilon$.

Now there exists $T_0$ such that $f(T_0) < 1$. The chain rule implies that $f(t+s) \leq f(t)f(s)$ and by Lemma~\ref{lem:bound} $f(t) \leq \sqrt{1+\kappa^2}e^{\kappa t}$. These three facts imply the existence of $C>0, \lambda >0$ such that $f(t) \leq Ce^{-\lambda t}$. 
\end{proof}

\subsection{No focal points}\label{subsec:no_focal}

In this subsection we prove the following proposition.

\begin{proposition}
\label{prop:NPC}
Let $M$ be a compact asymptotically harmonic manifold. If $M$ has no focal points, then $M$ is either flat or a rank one symmetric space of noncompact type.
\end{proposition}

If $M$ is a Riemannian manifold and $v \in SM$ then the \emph{rank} of $v$ is the dimension of the vector space of bounded Jacobi fields along the geodesic $\gamma_v(t)$. Then define the \emph{rank} of $M$ to be minimum rank over all $v \in SM$. As the next theorem shows within the category of compact manifolds without focal points, those with rank one are ``generic.'' 

\begin{theorem}\cite{watkins11}
Suppose $M$ is a compact Riemannian manifold without focal points. Let $X$ be the universal Riemannian cover of $M$. If $M$ has rank greater or equal to 2, then $X$ is either a Riemannian product or a symmetric space.
\end{theorem}

The above theorem is true if compactness is replaced by a weaker condition, see the exposition in~\cite{watkins11} for details. The above theorem is a generalization of the well known Rank Rigidity Theorem of Ballmann~\cite{ballmann85} and Burns and Spatzier~\cite{BS87} for manifolds with nonpositive curvature. 

For manifolds without focal points $U^u(v)=-U^s(-v)=B(-v)$~\cite[Theorem 1, Proposition 5]{eschen77} and by Lemma~\ref{lem:kerV} a vector $v \in SM$ has rank one if and only if
\begin{align*}
0 \neq \det(U^u(v)+U^s(v) ) = \det (B(-v)+B(v) ).
\end{align*}
So by Proposition~\ref{prop:rank_one}, it is enough to show that any compact asymptotically harmonic manifold without focal points and higher rank is flat. Notice that a manifold is asympototically harmonic if and only if its universal cover is asymptotically harmonic. By the Rank Rigidity theorem we know that any such manifold must be a product or an irreducible higher rank symmetric space. 

Ledger~\cite{ledger57} showed that irreducible symmetric harmonic manifolds must have rank one (see also Eschenburg~\cite{eschenburg80}). The same is true for asymptotically harmonic manifolds, for instance if $X$ is a non-compact symmetric space the discussion in~\cite{eschenburg80} implies that
\begin{align*}
\tr U^s(v) = -\sum_{\alpha \in Rt} k_{\alpha} \abs{\alpha(v)}
\end{align*}
where $Rt$ is the set of roots of $X$ and $k_{\alpha}$ are nonnegative integers. If rank is greater or equal to two and $X$ is irreducible, the right hand side will not be constant.

\begin{lemma} 
\label{lem:no_higher_rank}
Suppose $X$ is an irreducible symmetric space of noncompact type. Then $X$ is asymptotically harmonic if and only if $X$ has rank one.
\end{lemma} 

It is a result of Lichnerowicz~\cite{lichnerowicz44} that a harmonic manifold which is a Riemannian product must be flat. The following two lemmas show that the same is true for asymptotically harmonic manifolds without focal points.

\begin{lemma}
\label{lem:prod1}
Supppose $X=X_1 \times X_2$ is a Riemannian product and $X$, $X_1$, $X_2$ have no conjugate points. Let $U^s, U_1^s, U_2^s$ be the stable Riccati solutions for $X,X_1,X_2$ respectively. If $X$ is asymptotically harmonic then $X_1$ and $X_2$ are asymptotically harmonic and $\tr U^s = \tr U_1^s = \tr U_2^s =0$.
\end{lemma}

\begin{proof} 
Suppose $X$ is asymptotically harmonic with $\tr U^s \equiv \alpha$. Viewing $T_{(x_1,x_2)}X = T_{x_1} X_1 \times T_{x_2} X_2$ and picking $(v_1,v_2) \in SX$ we have
\begin{align}
\label{eq:prod_lemma}
\tr U^s(v_1,v_2) = \norm{v_1} \tr U_1^s(v_1/\norm{v_1}) + \norm{v_2} \tr U_2^s(v_2/\norm{v_2}).
\end{align}
To see this let $\nabla,\nabla^1,\nabla^2$ be the Levi-Civita connections on $X,X_1,X_2$ respectively. Then
\begin{align*}
\nabla_{(Y_1,Y_2)} (Z_1,Z_2) = (\nabla^1_{Y_1} Z_1, \nabla^2_{Y_2} Z_2)
\end{align*}
implying, using the notation in Section 2, that
\begin{align*}
\JJ_T(t) = \left(\JJ^1_{\norm{v_1}T}\left(\norm{v_1}t\right), 
\JJ^2_{\norm{v_2}T}\left(\norm{v_2}t\right)\right)
\end{align*}
where $\JJ_T,\JJ^1_T,\JJ^2_T$ are the Jacobi tensors along the geodesics $\gamma_{(v_1,v_2)}$ in $X$, $\gamma_{v_1/\norm{v_1}}$ in $X_1$, $\gamma_{v_2/\norm{v_2}}$ in $X_2$ respectively that are equal to the identity at $t=0$ and vanish at $t=T$. Then taking the limit as $T\rightarrow \infty$ of $\JJ^\prime_T(0)$ yields equation~\ref{eq:prod_lemma}.

So if $v_1 \in SX_1$ and $v_2 \in SX_2$, plugging in $(v_1,0)$ and $(0,v_2)$ into Equation~\ref{eq:prod_lemma} yields $\tr U_1^s(v_1) = \alpha$ and $\tr U_2^s(v_2) \equiv \alpha$. But plugging in $(v_1/\sqrt{2},v_2/\sqrt{2})$ in Equation~\ref{eq:prod_lemma} yields 
\begin{align*}
\alpha = \tr U^2(v_1/\sqrt{2},v_1/\sqrt{2}) = \sqrt{2} \alpha
\end{align*}
and so $\alpha =0$.
\end{proof}

\begin{lemma}
\label{lem:prod2}
Let $X$ be an asymptotically harmonic manifold without focal points and $\tr U^s \equiv 0$, then $X$ is flat.
\end{lemma}

\begin{proof}
As $X$ has no focal points $U^s$ is negative semidefinite (see for instance see the proof of Corollary 3.3 in Eberlein~\cite{eberlein73I}), further $\tr U^s  \equiv 0$ thus as $U^s$ is symmetric, $U^s \equiv 0$ and then the Riccati equation 
\begin{align*}
(U^s)^\prime + (U^s)^2+R=0
\end{align*}
implies that $R \equiv 0$.
\end{proof}

\subsection{Purely Exponential Volume Growth}\label{subsec:purelyExp}

In this subsection we establish the following.

\begin{proposition}
\label{prop:purelyExp}
Let $M$ be a compact asymptotically harmonic manifold with universal Riemannian cover $X$. If $X$ has purely exponential volume growth, then $X$ is a rank one symmetric space of noncompact type.
\end{proposition} 

Using a result of Coornaert~\cite[Theorem 7.2]{coor93} any Gromov hyperbolic simply-connected manifold $X$ with a compact Riemannian quotient has purely exponetial volume growth, so we obtain the following.

\begin{corollary}
Let $M$ be a compact asymptotically harmonic manifold with universal Riemannian cover $X$. If $X$ is Gromov hyperbolic, then $X$ is a rank one symmetric space of noncompact type.
\end{corollary}

By Proposition~\ref{prop:rank_one}, it is enough to show that there exists $v \in SX$ with $\det(B(-v)+B(v)) \neq 0$ . By Proposition~\ref{prop:asym_constant} $B(-v) = U^u(v)$ and $-B(v)=U^s(v)$ on a set of full measure in $SM$, so it is actually enough to show that $\det(U^u(v)-U^s(v)) \neq 0$ on a set of positive measure. To prove this we will exploit a connection between rank and volume growth that was used by Knieper~\cite{knieper11} in the context of harmonic manifolds.

We begin by stating a lemma from linear algebra.

\begin{lemma}
\label{lem:det}
Suppose $A,B$ are positive definite matrices and $A \leq B$, then $\det A \leq \det B$.
\end{lemma}

\begin{proof}
By Minkowski's inequality
\begin{equation*}
\det(B)^{1/n} =\det(A+(B-A))^{1/n} \geq \det(A)^{1/n} + \det(B-A)^{1/n} \geq \det(A)^{1/n}. \qedhere
\end{equation*}
\end{proof}

Next we recall some useful facts about volume growth in a Riemmannian manifold. Let $X$ be a simply connected Riemannian manifold and $p \in X$ then
\begin{align*}
\text{Vol}_X \ B_r(p) = \int_0^r \int_{S_pM} \det A_v(t) dv dt
\end{align*}
where $dt,dv$ are the standard Lebesque measures on $\RR$ and $S_pX$ and $A_v(t)$ is the Jacobi tensor along the geodesic $\gamma_v$ with initial conditions
\begin{align*}
A_v(0)=0 \text{ and } A_v^\prime(0)=Id.
\end{align*}
The tensor $A_v(t)$ can be written in terms of the stable and unstable Jacobi tensors as: 
\begin{align*}
A_v(t) = \Big(\JJ_v^u(t)-\JJ^s_{v,T}(t)\Big)\Big((\JJ_v^u)^\prime(0)-(\JJ^s_{v,T})^\prime(0)\Big)^{-1}
\end{align*}
for each $T >0$. To see this last assertion, notice that both sides of the above equation are Jacobi tensors and have the same initial conditions. In particular, 
\begin{align*}
A_v(t) & = \Big(\JJ_v^u(t)-\JJ^s_{v,t}(t)\Big)\Big((\JJ_v^u)^\prime(0)-(\JJ^s_{v,t})^\prime(0)\Big)^{-1} \\
 & = \JJ_v^u(t)\Big((\JJ_v^u)^\prime(0)-(\JJ^s_{v,t})^\prime(0)\Big)^{-1}.
\end{align*}

Now Proposition~\ref{prop:purelyExp} follows from the next Lemma.  

\begin{lemma}
If $X$ is a simply connected asymptotically harmonic manifold and $p \in X$ such that
\begin{align*}
\det(U^u(v)-U^s(v)) =0
\end{align*}
for a set of positive Lebesque measure in $S_pX$ then
\begin{align*}
\lim_{r \rightarrow \infty} \frac{ \text{Vol}_X \ B_r(p)}{e^{hr}} = +\infty.
\end{align*}
\end{lemma}

\begin{proof}
By Proposition~\ref{prop:asym_constant}, $h_{vol}=\Delta b_v(p) = \tr \nabla^2 b_v(p)$ for all $v  \in S_pX$. Further by Lemma~\ref{lem:ineq1}, $U^u(v) \geq B(-v) = \nabla^2 b_{-v}(p)|_{v^{\bot}}$ so
\begin{align*}
\frac{d}{dt} \log \det \JJ^u_v(t) = \tr \ (\JJ^u_v)^{-1}(t)(\JJ^u_v)^\prime(t)=\tr U^u(g^tv) \geq \tr \nabla^2 b_{-g^tv}(\gamma_v(t))=h_{vol}.
\end{align*}
Since $\JJ^u_v(0)=Id$, we obtain $e^{h_{vol} t} \leq \det \JJ^u_v(t)$. Now for each $v \in S_pX$, by Proposition~\ref{prop:riccati_basic},
\begin{align*}
\Big((\JJ_v^u)^\prime(0)-(\JJ^s_{v,t})^\prime(0)\Big) \rightarrow \Big(U^u(v)-U^s(v)\Big)
\end{align*}
and the convergence is monotone. In particular by Lemma~\ref{lem:det}  
\begin{align*}
g(t,v) = \det\Big((\JJ_v^u)^\prime(0)-(\JJ^s_{v,t})^\prime(0)\Big)
\end{align*}
converges monotonically to $\det (U^u(v)-U^s(v))$. Fix $R_0>0$ then for $r >R_0$ we have:
\begin{align*}
\text{Vol}_X \ B_r(p) & = \text{Vol}_X \ B_{R_0}(p)+ \int_{R_0}^r \int_{S_pX} \frac{\det \JJ_v^u(t)}{g(t,v)} dv dt \\
& \geq \text{Vol}_X \ B_{R_0}(p)+ \int_{R_0}^r \int_{S_pX} \frac{e^{h_{vol}t}}{g(R_0,v)} dv dt \\
& = \text{Vol}_X \ B_{R_0}(p) + \frac{e^{h_{vol} r} - e^{h_{vol} R_0}}{h_{vol}} \int_{S_pX} \frac{dv}{g(R_0,v)}
\end{align*}
Now dividing by $e^{h_{vol}r}$ and sending $r \rightarrow \infty$ yields
\begin{align*}
\liminf_{r \rightarrow \infty} \frac{ \text{Vol}_X \ B_r(p)}{e^{hr}} \geq \frac{1}{h_{vol}} \int_{S_pX} \frac{dv}{g(R_0,v)}.
\end{align*}
As $R_0>0$ was arbitrary, the Lebesque monotone convergence theorem implies that
\begin{align*}
(h_{vol})\liminf_{r \rightarrow \infty} & \frac{ \text{Vol}_X \ B_r(p)}{e^{hr}}  \geq \lim_{R_0\rightarrow \infty} \int_{S_pX} \frac{dv}{g(R_0,v)} = \int_{S_pX} \lim_{R_0\rightarrow \infty} \frac{1}{g(R_0,v)} dv.
\end{align*}
Finally as 
\begin{align*}
\lim_{R_0 \rightarrow \infty} g(R_0,v) = \det (U^u(v)-U^s(v))
\end{align*}
the lemma follows.
\end{proof}

\section{The Busemann compactification and Patterson-Sullivan measures}

We begin by recalling the Busemann compactification of a non-compact complete Riemannian manifold $X$. As in~\cite{Led10, LW10} we normalize our Busemann functions such that $\xi(o)=1$ for some point $o \in X$.

Fix a point $o \in X$ and for each $y \in X$ define the Busemann function based at $y$ to be the function
\begin{align*}
b_y(x) = d(x,y)-d(y,o).
\end{align*}
As each $b_y$ is 1-Lipschitz, the embedding $y \rightarrow b_y \in C(X)$ is relatively compact if $C(X)$ is equipped with the topology of uniform convergence on compact subsets. We then define the \emph{Busemann compactification} $\hat{X}$ of $X$ to be the compactification of $X$ in $C(X)$. The \emph{Busemann boundary} of $X$ is the set $\partial \hat{X} = \hat{X} \setminus X$. We begin by recalling some features of this compactification. If $(X,g)$ has no conjugate points for each $v \in S_oX$ there is a natural Busemann function
\begin{align*}
b_v(x) = \lim_{t \rightarrow \infty} d(x,\gamma_v(t))-t.
\end{align*}

\begin{theorem}
\label{thm:buse_bd_basic}
Let $X$ be a non-compact complete simply connected Riemannian manifold. Then
\begin{enumerate}
\item $X$ is open in $\hat{X}$, hence the Busemann boundary $\partial \hat{X}$ is compact.
\item The action of $\text{Isom}(X)$ on $X$ extends to an action on $\hat{X}$ by homeomorphisms and for $\gamma \in \text{Isom}(X)$ and $\xi \in \partial \hat{X}$ the action is given by 
\begin{align*}
(\gamma \cdot \xi)(x) = \xi(\gamma^{-1}x)-\xi(\gamma^{-1}o).
\end{align*}
\item If $X$ has no conjugate points, then for $v \in SX$ each $b_v$ is $C^1$
\item If $X$ has no conjugate points and sectional curvature bounded below, then each $\xi \in \partial \hat{X}$ is $C^1$ and $\norm{\nabla  \xi}=1$. In particular the integral curves of $\xi$ are geodesics.
\end{enumerate}
\end{theorem}

The first result can be found in~\cite[Proposition 1]{LW10}. The second assertion is straightforward to prove. The third assertion can be found in~\cite[Proposition 1]{eschen77}. It remains to prove the fourth assertion.

\begin{proof}
Suppose
\begin{align*}
b_{y_n}(x) = d(x,y_n)-d(y_n,o)
\end{align*}
converges locally uniformly to a function $\xi \in \partial \hat{X}$. Fix a compact set $K\subset X$. By~\cite[Lemma 2.8]{eberlein73I} there exists $\kappa >0$ such that for $d(x,y_n)>1$ we have $\norm{\nabla^2 b_{y_n}(x)} \leq \kappa$. So by the Arzel\`{a}-Ascoli theorem and by passing to a subsequence we may assume $\nabla b_{y_n}$ converges uniformly to a continuous vector field on $K$. This implies that $\xi$ is $C^1$ on $K$. As $K$ was arbitrary, this implies that $\xi$ is $C^1$. It is then straightforward to verify that $\norm{\xi}=1$, see for instance~\cite[Lemma 1]{LW10}.
\end{proof}

\begin{proposition}
\label{prop:busemann_fcns}
Suppose $X$ is a simply connected manifold without conjugate points. If the map
\begin{align*}
v \in SX \rightarrow \Phi(v)=b_v \in C(X)
\end{align*}
is continuous, then the map
\begin{align*}
v \in S_oX \rightarrow \Phi_o(v)=b_v \in \partial \hat{X}
\end{align*}
it is a homeomorphism. 
\end{proposition}

\begin{remark} Because of our choice of normailzation unless $v\in S_oX$ the function $b_v$ will not be in $\partial \hat{X}$. \end{remark}

\begin{proof}
Theorem~\ref{thm:buse_bd_basic} shows that each $b_v$ is $C^1$ and $\nabla b_v(o)=-v$ so $\Phi_o$ is one to one. Now let
\begin{align*}
b_v^t(x) = d(x,\gamma_v(t))-t.
\end{align*}
The triangle inequality shows that $b_v^{T_1}(x) \geq b_v^{T_2}(x)$ for $T_1 < T_2$. As the map $(x,v) \rightarrow b_v(x)$ is continuous, Dini's Theorem then implies that the convergence $b_v^t(x) \rightarrow b_v(x)$ is locally uniform in $v \in SX$ and $x \in X$. Now suppose $b_{y_n}$ converges to a Busemann function $\xi \in \partial \hat{X}$, then using the completeness of $X$ there exists $v_n \in S_oX$ and $t_n \in [0,\infty)$ such that $b_{y_n} = b_{v_n}^{t_n}$. By passing to a subsequence we can assume $v_n \rightarrow v$ and then using the fact that the convergence $b_w^t(x) \rightarrow b_w(x)$ is uniform in $w \in S_oX$ and $x \in X$  we see that $\xi = b_v$. This show that $\Phi_o$ is onto. Finally as $S_oX$ is compact, $\Phi_o:S_oX \rightarrow \partial \hat{X}$ is a homeomorphism.  
\end{proof} 

In particular, using Theorem~\ref{thm:cont} we have the following.

\begin{corollary}
\label{cor:bd_identification}
Suppose $X$ is an asymptotically harmonic manifold, then the Busemann boundary consists of functions of the form
\begin{align*}
b_v(x) = \lim_{t \rightarrow \infty} d(x,\gamma_v(t)) - t \text{ for $v \in S_oX$}.
\end{align*}
\end{corollary}

Ledrappier and Wang~\cite{LW10} considered Patterson-Sullivan measures on the Busemann boundary of general non-compact Riemannian manifolds. Let $M$ be a compact Riemannian manifold with non-compact universal Riemannian cover $X$ and let $\Gamma=\pi_1(M) \subset \text{Isom}(X)$ be the deck transformations of the covering $X \rightarrow M$.  First recall the \emph{volume growth entropy of $X$}:
\begin{align*}
h_{vol} = \lim_{R\rightarrow \infty} \frac{\log \text{Vol}_X \ B_R(p)}{R}.
\end{align*}
Manning~\cite{manning79} showed that the limit above always exists and is independent of $p$. Further Manning showed that when $M$ is nonpositively curved $h_{top}=h_{vol}$, where $h_{top}$ is the topological entropy of the geodesic flow on $SM$. Freire and Ma\~{n}\'{e}~\cite{FM82} generalized this last result and showed that $h_{top}=h_{vol}$ when $M$ has no conjugate points. 

\begin{definition}
Suppose $M$ is a compact Riemannian manifold. Let $X$ be the universal Riemannian cover of $M$ with deck transformations $\Gamma=\pi_1(M) \subset \text{Isom}(X)$. A family of measures $\{ \nu_x : x \in X\}$ on $\partial \hat{X}$ is a (normalized) \emph{$\Gamma$-Patterson-Sullivan measure} if
\begin{enumerate}
\item $\nu_o(\partial \hat{X})=1$,
\item for any $x,y \in X$ the measures $\nu_x,\nu_y$ have the same measure class and satisify 
\begin{align*}
\frac{d\nu_x}{d\nu_y}(\xi) = e^{-h_{vol}(\xi(x)-\xi(y))},
\end{align*}
\item for any $g \in \Gamma$, $\nu_{gx} = g_*\nu_x$.
\end{enumerate}
\end{definition}

With the notation in the above definition, consider the laminated space 
\begin{align*}
X_M = (X \times \partial \hat{X}) / \Gamma
\end{align*}
where $\Gamma$ acts diagonally on the product. Then $X_M$ is foliated by the images of $X \times \{ \xi \}$ under the projection. The leaves of this foliation inherit a smooth structure from $X$ and using this structure we can define a gradient $\nabla^{\WW}$, a divergence $\text{div}^{\WW}$, and a Laplacian $\Delta^{\WW}$ in the leaf direction. A Patterson-Sullivan measure $\{\nu_x : x \in X\}$ yields a finite measure on the laminated space $X_M$ as follows: by definition $d\nu_x(\xi)=e^{-h_{vol}\xi(x)}d\nu_o(\xi)$ for all $x \in X$. In particular if $dx$ is the Riemannian volume form on $M$, then the measure
\begin{align*}
d\wt{m}(x,\xi)=e^{-h_{vol}\xi(x)}dxd\nu_o(\xi)
\end{align*}
is $\Gamma$-invariant and descends to a measure $m$ on $X_M$. With this notation, Ledrappier and Wang proved the following integral formula for the volume growth entropy.

\begin{theorem}
\cite[Theorem 1]{LW10}
With the notation above, there exists a $\Gamma$-Patterson-Sullivan measure $\{ \nu_x : x \in X\}$ on the Busemann boundary $\partial \hat{X}$. Further, for any such measure and any continuous vector field $Y$ on $X_M$ which is $C^1$ along the leaves $X \times \{\xi\}$,
\begin{align*}
\int \text{div}^\mathcal{W} Y dm = h_{vol}\int \left\langle Y,\nabla^{\mathcal{W}} \xi  \right\rangle dm.
\end{align*}
\end{theorem}  

This should be compared to a result of Freire and Ma\~{n}\'{e}.

\begin{theorem}
\cite[Equation 9]{FM82}
Let $M$ be a compact Riemannian manifold without conjugate points. If $d\lambda$ is the Liouville measure on $SM$ with total volume one then
\begin{align*}
h_{\lambda}(g^t) = \int \tr U^u(v) d\lambda(v)
\end{align*}
where $h_{\lambda}(g^t)$ is the measure theoretic entropy of the geodesic flow on $SM$ with respect to the measure $\lambda$.
\end{theorem}

As a corollary to these results we obtain the following.

\begin{corollary}
\label{cor:asym_constant}
Let $M$ be a compact asmptotically harmonic manifold with universal Riemannian cover $X$ and $\Delta b_v \equiv \alpha$ for all $v \in SX$. Then $\alpha = h_{vol}$ and $B(-v)= U^u(v)$ for almost every $v \in SM$ with respect to the Liouville measure.    
\end{corollary}

\begin{proof}
As $M$ is asymptotically harmonic, $\partial \hat{X} = \{ b_v : v \in S_oX\}$. 
By Theorem~\ref{thm:cont} the map $(x,\xi)\in X \times \partial \hat{X} \rightarrow \nabla \xi(x)$ is continuous and descends to a vector field on $X_M = X \times \partial \hat{X}/\Gamma$. 
In particular, the integral formula of Ledrappier and Wang implies that
\begin{align*}
h_{vol}\cdot m(X_M) = h_{vol} \int \norm{\nabla^{\WW} \xi} dm  = \int \text{div}^{\WW} \nabla^{\WW} \xi dm =\int \Delta^{\WW} \xi dm =\alpha \cdot m(X_M)
\end{align*}
so $\alpha = h_{vol}$. However by Lemma~\ref{lem:ineq1}, $B(-v) \leq U^u(v)$ and the integral formula of Freire and Ma\~{n}\'{e} implies that
\begin{align*}
h_{vol} = \int \tr B(-v) d\lambda(v) \leq \int \tr U^u(v) d\lambda(v) = h_{\lambda}(g^t).
\end{align*}
Now let $h_{top}$ be the topological entropy of the geodesic flow on $SM$. By the variational formula, $h_{\lambda}(g^t) \leq h_{top}$. Further, Freire and Ma\~{n}\'{e}~\cite{FM82} proved that $h_{vol} = h_{top}$ if $M$ is compact and has no conjugate points. Summarizing, we have that
\begin{align*}
h_{vol} = \int \tr B(-v) d\lambda(v) \leq \int \tr U^u(v) d\lambda(v) \leq h_{vol}.
\end{align*}
The second assertion of the corollary then follows. 
\end{proof}
 
\section{Harmonic Measures}

As in Section 4, suppose $M$ is a compact Riemannian manifold. Let $X$ be the universal Riemannian cover of $M$ with deck transformations $\Gamma=\pi_1(M) \subset \text{Isom}(X)$. Again consider the Busemann compactification $\hat{X} = X \sqcup \partial\hat{X}$ and the laminated space 
\begin{align*}
X_M = (X \times \partial \hat{X}) / \Gamma
\end{align*}
where $\Gamma$ acts diagonally on the product. A Borel measure $m$ on $X_M$ is said to be \emph{harmonic} if 
\begin{align*}
\int \Delta^{\WW} f dm =0
\end{align*}
for all $f$ continuous on $X_M$ and $C^2$ along each leaf. 

\begin{theorem}\cite[Theorem 1]{garnett83}
With the notation above, let $m$ be a harmonic measure on $X_M$ and $\wt{m}$ the $\Gamma$-invariant lift of $m$ to a measure on $X \times \partial \hat{X}$. Then there is a finite measure $\mu$ on $\partial \hat{X}$ and for $\mu$-almost every $\xi \in \partial \hat{X}$ there is a positive harmonic function $k_\xi$ on $X$ with $k_\xi(o)=1$ such that
\begin{align*}
d\wt{m}(x,\xi) = k_{\xi}(x) dx \times d\mu(\xi)
\end{align*}
where $dx$ is the standard Riemannian volume on $X$.
\end{theorem}

In this context results of Ledrappier imply the following.

\begin{theorem}
\cite{Led10}
\label{thm:param_eq}
Suppose $M$ is a compact Riemannian manifold without conjugate points. Let $X$ be the universal Riemannian cover of $M$ with deck transformations $\Gamma=\pi_1(M) \subset \text{Isom}(X)$. If $4\lambda_{min}=h_{vol}^2>0$ then there exists a $\Gamma$-Patterson-Sullivan measure $\{ \nu_x : x \in X\}$ on $\partial \hat{X}$ such that the measure
\begin{align*}
d\wt{m} = dx \times d\nu_x(\xi)=e^{-h_{vol}\xi(x)}dx \times d\nu_o(\xi)
\end{align*}
on $X \times \partial \hat{X}$ descends to a harmonic measure on the laminated space $X_M$. In particular for $\nu_o$-almost every $\xi$, $\xi \in C^\infty(X)$ and $\Delta \xi \equiv h_{vol}$.
\end{theorem}

We will quickly show how Ledrappier's results can be combined to obtain Theorem~\ref{thm:param_eq}.

\begin{proof}
If $4\lambda_{min}=h_{vol}^2$ then Corollary 0.2, Proposition 1.1, 1.2, and 1.3 in~\cite{Led10} imply the existence of a measure $m$ on the laminated space 
\begin{align*}
Y_M = (X \times \hat{X}) / \Gamma,
\end{align*}
a finite measure $\nu$ on $\hat{X}$, and for $\nu$-almost every $\xi \in \hat{X}$ a positive harmonic function $k_\xi$ on $X$ with $k_\xi(o)=1$ such that the $\Gamma$-invariant lift of $m$ to $X \times \hat{X}$ is represented by
\begin{align*}
\wt{m} = k_{\xi}(dx \times \nu(d\xi)).
\end{align*}
Further outside a set of $\wt{m}$ measure zero we have $\nabla_x \ln k_{\xi}(x) = -h_{vol} \nabla_x \xi(x)$. By scaling we may suppose $\nu$ is a probability measure. Notice that $\wt{m}$ is a measure on $X \times \hat{X}$ instead of $X \times \partial \hat{X}$. 

By the Fubini theorem, there exists a set $\Omega \subset \hat{X}$ such that $\nu(\Omega)=1$ and for every $\xi \in \Omega$ 
\begin{align*}
\nabla_x \ln k_{\xi}(x) = -h_{vol} \nabla_x \xi(x)
\end{align*}
outside a set of measure zero in $X$. As $k_\xi$ is harmonic and positive, $\ln k_\xi$ is $C^\infty$ and the map $x \rightarrow \nabla \ln k_\xi(x)$ is continuous. Now suppose $\xi \in \Omega$. If $\xi \notin \partial \hat{X}$, then $\xi = b_y$ for some $y \in X$. This implies that $x \rightarrow \nabla \xi(x)$ is continuous on $X\setminus\{y\}$. So $\nabla_x \ln k_{\xi}(x) = -h_{vol} \nabla_x \xi(x)$ for $x \in X \setminus \{y\}$, which implies by integration that $\ln k_{\xi} = -h_{vol} \xi$ on $X$. But this is nonsense since $\xi(x) = b_y(x)$ is not differentiable at $x=y$.

So $\nu$ is supported on $\partial \hat{X}$. This implies that $\wt{m}$ can be realized as a measure on $X\times \partial \hat{X}$ and descends to a harmonic measure on $X_M$.

Now for $\xi \in \partial \hat{X}$ Theorem~\ref{thm:buse_bd_basic} implies that the map $x \rightarrow \nabla \xi(x)$ is continuous. So for $\xi \in \Omega$ we have $\nabla_x \ln k_{\xi}(x) = -h_{vol} \nabla_x \xi(x)$ for all $x \in X$ which implies that $k_\xi = e^{-h_{vol} \xi}$ and that the measures
\begin{align*}
d\nu_x(\xi) = k_{\xi}(x)d\nu(\xi) = e^{-h_{vol}\xi(x)} d\nu(\xi)
\end{align*}
form a $\Gamma$-Patterson-Sullivan measure on $\partial \hat{X}$.
\end{proof}

\section{Proof of Theorem~\ref{thm:visib}}

A simply connected complete Riemannian manifold $(X,g)$ without conjugate points is called a (uniform) \emph{visibility manifold} if given $\epsilon >0$ there exists $R=R(\epsilon) >0$ such that for every $p,x,y \in X$ such that the (unique) geodesic $[x,y]$ is a distance greater than $R$ from $p$, then the geodesics $[p,x]$ and $[p,y]$ make an angle less than $\epsilon$ at $p$. In short, geodesics far from $p$ look small. The basic example is the universal cover of a compact negatively curved manifold. We begin by recalling a theorem of Ruggiero.

\begin{theorem}
\cite[Theorem 6.8]{ruggiero07}
Let $M$ be a compact Riemannian manifold without conjugate points and let $X$ be the universal Riemannian cover of $M$. Then $X$ is a visibility manifold if and only if $X$ is Gromov hyperbolic and geodesics diverge in $X$.
\end{theorem}

In a non-compact complete Riemannian manifold $X$, geodesics are said to diverge if given any $p \in X$ and a distinct pair $v,w \in S_pX$ then
\begin{align*}
\lim_{t \rightarrow+\infty} d(\gamma_v(t),\gamma_w(t)) = +\infty.
\end{align*}

The purpose of this section is to prove Theorem~\ref{thm:visib} which we recall:

\begin{theorem} 
Suppose $M$ is a compact Riemannian manifold without conjugate points. Let $X$ be the universal Riemannian cover of $M$ with deck transformations $\Gamma=\pi_1(M) \subset \text{Isom}(X)$. If $(X,g)$ is a visibility manifold, then the following are equivalent
\begin{enumerate}
\item $X$ is a rank one symmetric space of noncompact type, 
\item $X$ is asymptotically harmonic,
\item each Busemann function $b_v$ is $C^2$ and $\Delta b_v \equiv h_{vol}$,
\item there exists a function $f:X \rightarrow \RR$ that is 1-Lipschitz and has $\Delta f \geq h_{vol}$ (in the sense of distributions),
\item $4\lambda_{min} = h_{vol}^2$,
\item there exists a $\Gamma$-Patterson-Sullivan measure $\{ \nu_x : x \in X\}$ such that the measure
\begin{align*}
d\wt{m} = dx \times d\nu_x(\xi)
\end{align*}
on $X \times \partial \hat{X}$ descends to a harmonic measure on the laminated space $X_M$.
\end{enumerate}
\end{theorem}

\begin{remark}
Theorem~\ref{thm:main1} shows that (1) and (2) are equivalent. Proposition~\ref{prop:asym_constant} shows that (2) and (3) are equivalent. Clearly (3) implies (4) as $\norm{\nabla b_v}=1$. The implication (4) implies (5) is due to Grigor\'yan~\cite[Theorem 11.17]{grig09} whose proof we provide below. Theorem~\ref{thm:param_eq} used results of Ledrappier~\cite{Led10} to show that (5) implies (6). It remains to prove that (6) implies (1).
\end{remark}

We begin by recalling Grigor\'yan's argument that (4) implies (5).

\begin{proposition}\cite[Theorem 11.17]{grig09}
Let $M$ be a compact manifold without conjugate points and let $X$ be the universal Riemannian cover of $M$.  Suppose that there exists a function $f:X \rightarrow \RR$ that is 1-Lipschitz and has $\Delta f \geq h_{vol}$ (in the sense of distributions), then $4\lambda_{min}=h_{vol}^2$.
\end{proposition}

\begin{proof}
For the reader's convenience we provide Grigor\'yan proof for the case when $f$ is $C^2$. Then for $\phi \in C_K^{\infty}(X)$,
\begin{align*}
h_{vol} \int_{X} 
&\phi^2 dx \leq \int_X \Delta f \phi^2 dx = -2\int_X \ip{\nabla f,\nabla \phi}\phi dx \\
&\leq 2 \int_X \norm{\nabla \phi} \phi dx \leq 
2 \left(\int_X \norm{\nabla \phi}^2 dx\right)^{1/2} \left(\int_X \phi^2 dx\right)^{1/2}.
\end{align*}
So we have for all $\phi \in C_K^{\infty}(X)$ that
\begin{align*}
h_{vol}^2 \leq 4\frac{\int_X \norm{\nabla \phi}^2 dx}{\int_X \phi^2 dx}.
\end{align*}
This implies that $h_{vol}^2 \leq 4\lambda_{min}$ and the reverse inequality is always true by the discussion proceeding Equation~\ref{eq:inq_vol_eigen}. 
\end{proof}

We now prove that (6) implies that $(M,g)$ is asymptotically harmonic.

\begin{proposition}
\label{prop:char}
Suppose $M$ is a compact Riemannian manifold without conjugate points. Let $X$ be the universal Riemannian cover of $M$ with deck transformations $\Gamma=\pi_1(M) \subset \text{Isom}(X)$. If $(X,g)$ is a visibility manifold and there exists a $\Gamma$-Patterson-Sullivan measure $\{ \nu_x : x \in X\}$ on $\partial \hat{X}$ such that the measure
\begin{align*}
d\wt{m} = dx \times d\nu_x(\xi)
\end{align*}
descends to a harmonic measure on the laminated space $X_M$, then $X$ is asymptotically harmonic.
\end{proposition}

We will need some facts Busemann functions in a Gromov hyperbolic space. Recall that a Gromov hyperbolic space has an \emph{ideal boundary} $X(\infty)$ defined to be equivalent classes of geodesics where two geodesics $\gamma$ and $\sigma$ are equivalent if 
\begin{align*}
\sup_{t \geq 0 } d(\gamma(t),\sigma(t)) < + \infty.
\end{align*}
Let $o \in X$ be the point in the definition of the Busemann boundary and let $[\sigma]$ denote the equivalence class of a geodesic $\sigma$. By using completeness we see that any equivalence class of geodesics can be represented by a geodesic $\gamma$ with $\gamma(0)=o$. If $\gamma^\prime(0)=v\in S_oX$ we will occasionally denote the Busemann function $b_v$ by $b_{\gamma}$:
\begin{align*}
b_{\gamma}(x)=b_v(x) = \lim_{t \rightarrow \infty} d(x,\gamma(t))-t.
\end{align*}
For Gromov hyperbolic spaces there is a nice relationship between the Busemann compactification and the ideal boundary.

\begin{theorem}\label{thm:GHbusemann_structure}\cite[Section 7.5.D]{Gromov87}
Let $X$ be a Gromov hyperbolic Riemannian manifold. Consider the quotient space $[\partial \hat{X}] = \partial \hat{X} / \sim$ where $\xi \sim \eta$ if and only if $\sup_{x \in X} | \xi(x)-\eta(x)| < +  \infty$. Then the map $\psi : X(\infty) \rightarrow [\partial \hat{X}]$ given by $\psi([\gamma]) = [b_{\gamma}]$ is a homeomorphism and the identity map $id : X \rightarrow X$ extends continuously to a map
\begin{align*}
X \sqcup \partial \hat{X} \rightarrow X \sqcup X(\infty)
\end{align*}
where the boundary is mapped by $\Psi=\psi^{-1} \pi$. Moreover, there exists $\kappa=\kappa(X)$ such that for all $\xi \sim \eta \in \partial \hat{X}$,
\begin{align*}
\sup_{x \in X} | \xi(x)-\eta(x)| < \kappa.
\end{align*}
\end{theorem}

Notice that we have choosen to normalize our Busemann functions such that $\xi(o)=0$. Otherwise the difference in the above theorem would have to be replaced with
\begin{align*}
\sup_{x,y \in X} \left| \Big(\xi(x)-\xi(y)\Big)-\Big(\eta(x)-\eta(y)\Big)\right| < \kappa.
\end{align*}

\begin{lemma}
Let $X$ be a Gromov hyperbolic Riemannian manifold. With the notation of Theorem~\ref{thm:GHbusemann_structure}: if $v \in SX$ and $\xi \in \partial \hat{X}$ and $[\gamma_v] \neq [\Psi(\xi)]$ then $\xi(\gamma_v(t)) \rightarrow +\infty$ as $t \rightarrow +\infty$. 
\end{lemma}
\begin{proof}
Let $u \in S_oX$ be such that $\xi \sim b_{u}$. Then $[\gamma_u] = [\Psi(\xi)]$ and so $[\gamma_u] \neq [\gamma_v]$. In particular,
\begin{align*}
d(\gamma_u(t),\gamma_v(t)) \rightarrow +\infty \text{ as $t \rightarrow +\infty$}.
\end{align*}
As $(X,g)$ is Gromov hyperbolic, this means that 
\begin{align*}
2t - d(\gamma_u(t),\gamma_v(t)) \leq C=C(v,u) < + \infty
\end{align*}
for some $C>0$ and all $t \geq 0$. So
\begin{equation*}
b_{u}(\gamma_v(t)) = \lim_{T \rightarrow \infty} d(\gamma_u(T),\gamma_v(t))-T \geq t +\lim_{T \rightarrow \infty} d(\gamma_u(T),\gamma_v(T))-2T \geq t - C(v,u). \qedhere
\end{equation*}
\end{proof}

Using the above results we will show that the ideal boundary and Busemann compactification coincide for visibility manifolds.

\begin{proposition}
\label{prop:buse_bd_visib}
Let $X$ be a simply connected visibility manifold without conjugate points and with sectional curvature bounded from below. Then
\begin{align*}
\partial \hat{X} = \{ b_v : v \in S_oM\}.
\end{align*}
\end{proposition}

\begin{proof}
As $X$ has no conjugate points and sectional curvature bounded from below, Theorem~\ref{thm:buse_bd_basic} implies that each $\xi \in \partial \hat{X}$ is $C^1$, $\norm{\nabla \xi}=1$, and the integral curves of $\xi$ are geodesics. By Theorem~\ref{thm:GHbusemann_structure}, there exists $v \in S_oM$ such that $\xi \sim b_v$. Now let $q \in X$ and suppose $\wt{v}=-\nabla b_v(q)$ and $w = -\nabla \xi(q)$. Then $b_v(\gamma_{\wt{v}}(t))=-t+b_v(q)$ and $\xi(\gamma_w(t))=-t+\xi(q)$, so by the Lemma above, we see that $[\gamma_{\wt{v}}] = [\gamma_v]=[\Psi(\xi)]=[\gamma_w]$. But $\wt{v},w \in S_qX$ and geodesics diverege in $(X,g)$, so we must have $\wt{v}=w$. As $q \in X$ was arbitrary, $\nabla b_v=\nabla \xi$ and hence that $b_v = \xi$.
\end{proof}

\begin{lemma}
Suppose $M$ is a compact Riemannian manifold without conjugate points. Let $X$ be the universal Riemannian cover of $M$ with deck transformations $\Gamma=\pi_1(M) \subset \text{Isom}(X)$. If $(X,g)$ is a visibility manifold and $\{ \nu_x : x \in X\}$ is a $\Gamma$-Patterson-Sullivan measure on $\partial \hat{X}$ then $\supp(\nu_x)=\partial \hat{X}$ for all $x \in X$.
\end{lemma}

\begin{proof}
Using Theorem~\ref{thm:GHbusemann_structure} and Proposition~\ref{prop:buse_bd_visib}, $\Psi:\partial \hat{X} \rightarrow X(\infty)$ is a homeomorphism and $\Psi_*\nu_x$ is a $\Gamma$-quasi-conformal density on $X(\infty)$ in the sense of Coornaert~\cite[Definition 4.1]{coor93}. Further, Corollary 5.2 in Coorneart's paper~\cite{coor93} shows that the support of any $\Gamma$-quasi-conformal density is the closure of $\Gamma x$ in $X(\infty)$. As $M=X / \Gamma$ is compact, this is all of $X(\infty)$.
\end{proof} 

\begin{proof}[Proof of Proposition~\ref{prop:char}]
Using the proof of Theorem~\ref{thm:param_eq} for $\nu_o$ almost every $\xi$, $\Delta \xi \equiv h_{vol}$. As $\partial \hat{X} = \supp(\nu_o)$, any set of full $\nu_o$ measure is dense in $\partial \hat{X}$. Further the set of functions $\{ f \in C^2(X): f(o)=1, \ \norm{\nabla f} \equiv 1 \ \Delta f \equiv h_{vol} \} \subset C(X)$ is closed, so for all $\xi \in \partial \hat{X}$, $\Delta \xi \equiv h_{vol}$. Let $v \in SX$ then
\begin{align*}
b_v(x) -b_v(o) = \lim_{t \rightarrow \infty} d(x,\gamma_v(t))-d(\gamma_v(t),o)
\end{align*}
and so $b_v-b_v(o) \in \partial \hat{X}$. Then $\Delta (b_v) = \Delta(b_v-b_v(o)) \equiv h_{vol}$ and $(X,g)$ is asymptotically harmonic. 
\end{proof}

\subsection*{Acknowledgements}

I would like thank Ralf Spatzier for many helpful conversations, Jordan Watkins for informing me of his recent generalization of the rank rigidity theorem, and Gerhard Knieper for pointing out an error in an earlier version of this paper. I also thank the National Science Foundation for support through the grant DMS-0602191.

\bibliographystyle{alpha}
\bibliography{geom}

\end{document}